\documentclass[12pt]{article}  

\usepackage{graphicx} 
\usepackage{epsfig} 
\usepackage{mathptmx} 
\usepackage{times} 
\usepackage{amsmath} 
\usepackage{amssymb}  
\usepackage{color}
\usepackage{amsthm}
\usepackage{enumerate}
\usepackage{subfigure}
\usepackage{caption}
\usepackage{multirow}
\usepackage{algorithm}
\usepackage{algpseudocode}
\usepackage{authblk}

\newtheorem{mydef}{Definition}
\newtheorem{mylem}{Lemma}
\newtheorem{myrem}{Remark}
\newtheorem{mythm}{Theorem}

\newtheorem{myprop}{Proposition}
\title{On node controllability and observability in complex dynamical networks.}
\author[1]{Francesco Lo Iudice}
\author[2]{Francesco Sorrentino}
\author[3]{Franco Garofalo}
\affil[1]{Department of Electrical Engineering and Information Technology,\\ University of Naples Federico II, Naples 80125, Italy}
\affil[2]{Department of Mechanical Engineering,\\ MSC01 1150
1 University of New Mexico, Albuquerque, NM 87131, USA}
\affil[3]{Department of Electrical Engineering and Information Technology,\\ University of Naples Federico II, Naples 80125, Italy}
\begin{document}
\maketitle

\begin{abstract}
We analyze in detail the subtle yet critical differences between the structural controllability and observability of the triplet $(A,B,C)$ in the two cases that this is viewed as a linear dynamical network of interconnected nodes or as a a single complex system. Investigating the controllability and observability properties of each single node when the network is not completely controllable and/or observable, we show that the first point of view requires the development of novel tools leading, ultimately, to a state space decomposition that is different from the one proposed in 1963 by R.E. Kalman for linear systems.
\end{abstract}

\section{Introduction}

The spectrum of real world systems that are modeled as complex dynamical networks is ever increasing, spanning from power grids, to financial networks  \cite{BuSp:09,PaAi:13,DeDiLo:18}. Our ability of controlling these networks towards a desired state is a topic that has attracted remarkable interest in the scientific community \cite{CoKa:13,LiJi:18,KlSh:17_b,BuDi:16}, leading researchers to tackle diverse problems such as ensuring complete network controllability \cite{LiSlBa:11,PePa:17,LiGo:18,LiDe:18,YuZh:13} or computing the minimal effort required to control a network \cite{pasqualetti2014controllability,yan2015spectrum,KlShSo:17}. A common trait among these studies is that of revisiting the fundamentals of dynamical systems theory to allow coping with large dynamical networks. 

Surprisingly, a fundamental tool that has been overlooked in these studies is the Kalman decomposition. By unveiling the portion of the state space that is made controllable by the system inputs and that is made observable through the available measurements, this tool gives the control designer a rather clear idea of the limitations to which the control action is subject. Hence, the following question naturally arises:
what insight can the Kalman decomposition provide on the controllability and observability of a large dynamical network? 
To give an answer to this question, first of all, we must consider that, often, a dynamic network develops autonomously and the need to control it arises at an advanced stage of its growth. Think for instance of power grids or traffic networks, which grow together with the cities, or nations, they serve or of a group of cells of a body organ that need a therapeutic interventions. Since such networks are not specifically designed to be controlled, two preliminary problems have to be solved. The first one is to establish in which nodes the control signals have to be injected. In the literature this is often referred to as the driver nodes selection problem \cite{LiSlBa:11,LiAl:17,LiGaSo,gaLi:14,PePr:17,ClBa:17}. The second problem is, obviously, the selection of the nodes that can be sensorized in order to get the measurements needed for the observation of the state of others nodes and for the synthesis of the control action. Differently from the first one, this second problem has attracted less attention from the researchers. When the number of actuators and sensors that can be deployed on the network is limited and the number of the nodes is large it can well be the case that the resulting network, seen as a linear dynamical system, lacks in complete controllability and/or complete observability. In this situation two problems naturally arise: (a) how can we find the \emph{set of nodes that can be controlled} and (b) how can we find the \emph{set of nodes whose state can be observed}. The first problem has been solved, see e.g. \cite{LiGaSo,gaLi:14,LiSlBa:12}, leveraging the structural approach proposed in \cite{Ho:80}, that is, leaving out of consideration the specific values of the network parameters. However, this solution highlights what apparently seems a contradiction as the set of controllable nodes depends only on the network structure, while it is well known that the controllable subspace of a dynamical system depends on the values of the system parameters. As for problem (b), a careful analysis of the literature shows that a solution is lacking, although, in systems theory, observability and controllability are geometrically dual concepts. 

In this paper, we will first clear up the apparent contradiction in the solution of problem (a) and then, extending the same reasoning, we will provide a solution to problem (b). In doing so, we will reach the striking conclusion that solving problems (a) and (b) does not boil down to finding the Kalman decomposition of the system state space. 
The reason is subtle but simple: following the Kalman approach, the controllability and observability properties are investigated through an ad hoc transformation of the system state representation. In the new basis the controllable and observable subsystems become visible but the physical meaning of the original system state is lost. 
When dealing with linear networks, instead, the process is somehow reversed. As the focus is on finding the states of the nodes that are controllable and observable, one must stick with the basis that associates a node to each of its elements, and then express the controllable and observable subnetworks through the elements of such basis. In turn, this constrains the transformations that can be used to perform the state space decomposition.  

Summing up, in this paper we show that the differences between what can be called the \textit{system state space decomposition} and the \textit{network state space decomposition} only emerge when we cope with partial controllability and observability. The new approach we propose in this paper will lead to the non uniqueness of the network state space decomposition and to the identification of  some interesting network subspaces: the one defined by the nodes that are not controlled but are perturbed by the control action and that generated by the intersection of the set of the observable system states and the network non observable subspace.

\section{Preliminaries}
We consider the linear ordinary differential equation
\begin{align}\label{eq:net_eq}
\dot{x} = Ax+Bu\\
y=Cx\nonumber
\end{align}
where the vectors $x\in \mathbb{R}^N$, $u\in\mathbb{R}^M$, and $y\in \mathbb{R}^P$. In this paper, we are going to consider the following two alternative interpretations of Eq. (1).
\newline \emph{Interpretation 1: Eq. \eqref{eq:net_eq} is a dynamical system.} The real matrix $A$ defines the system dynamics, the matrix $B$ represents the effect of the $M$ inputs in the vector $u$ on the state variables, and the matrix $C$ defines which $P$ linear combinations of the state variables are measured and thus consitute the output vector $y$.
\newline \emph{Interpretation 2: Eq. \eqref{eq:net_eq} is a dynamical network.} The real matrix $A = \lbrace a_{ij}\rbrace_{i,j=1}^N$ describes the node intrinsic dynamics and the network connectivity. Namely, the diagonal elements of the matrix $A$ define the node intrinsic dynamics, while if the $ij$-th element of $A$, $i\neq j$, is different from zero then there is an edge connecting node $v_j$ to node $v_i$. Accordingly, we define the graph associated to the matrix $A$, say $\mathcal{G}(A)$ as the set of nodes $\mathcal{V} = \lbrace v_1, \dots , v_N \rbrace$, and the set of edges $\mathcal{E}$, where $(i,j)\in \mathcal{E}$ iff $a_{ij}\neq 0$. In this paper, we will represent the intrinsic node dynamics as self loops in the graph $\mathcal{G}$, i.e., connections from a node to itself. The vector $u \in \mathbb{R}^M$ in eq \eqref{eq:net_eq} describes the $N_D$ input signals injected in a subset of the network nodes, the \emph{drivers}, identified by the matrix $B$; if the $ij$-th element of the matrix $B$ is different from zero, then the $j$-th input signal is injected in the $i$-th network node. Here, we assume that each one of the columns of the matrix $B$ encompasses only one nonzero entry \cite{LiGaSo}. Finally, the vector $y \in \mathbb{R}^P$ should be interpreted as the stack vector of the measured node states, that is, the state of the nodes where the sensors are placed (the \emph{sensor nodes}). Consistently, each row $c_i$ of the matrix $C$ is a versor with only one nonzero entry in the $j$-th position to indicate that node $v_j$ is a sensor node.

In what follows, we will make use of the following definition. 
\begin{mydef}\label{def:path}
We denote by $\pi_{ji}(k)$ the path of length $k$ from node $v_i$ to node $v_j$, that is, the sequence of $k$ edges $(i,r_1),(r_1,r_2),\dots, (r_{k-1},j)\rbrace$. Moreover, we define the weight of the path $\pi_{ji}(k)$
$$w_{ji}(k):=\prod_{(r_l,r_{l+1})\in\pi_{ji}(k) } a_{r_l,r_{l+1}}.$$
\end{mydef}
Next, we provide some background on the theory of structural controllability \cite{Lin:74}, \cite{Ho:80}, \cite{ShBP:75} . We start by defining an entry of a matrix as fixed, if its value is constrained to be zero, or free, if it can take an arbitrary value. Then, we can say that two matrices share the same structure if they share the positions of the fixed and free entries. This leads to introducing the concept of a structured matrix, that is, a matrix with fixed and free entries, the latter being indeterminates \cite{ShBP:75}. We are now ready to introduce the following result from generic analysis \cite{Fro:12,Ko:31}.
\begin{mylem}\label{lem:gen_an}
The generic rank of a matrix, that is, the rank the matrix takes for all selections of its free entries except for a set of Lebesgue measure zero, is equal to the maximal number of independent free entries of the matrix, where a set of free entries is said to be independent if no two lie on the same row, nor on the same column.
\end{mylem}
Note that a structured matrix is only endowed of a generic rank, while a matrix of which we not only know the structure, but also the values of the free entries is endowed of a rank and of a generic rank. The generic rank of a matrix coincides with the maximal rank a matrix with the same structure can take as we vary the values of its free entries.

Since the matrix $A$ in \eqref{eq:net_eq} can be interpreted as an adjacency matrix, so can its transpose $A^T$, which allows us to define the graph $\mathcal{G}^T$. Note that $\mathcal{G}^T$ corresponds to the network graph with reversed edges and thus, coherently with Interpretation 2, it is unequivocally defined by the structure of the matrix $A$. The observability matrix of the dynamical system \eqref{eq:net_eq} is defined as
\begin{equation}\label{eq:O_mat}
O = 
\left[\begin{array}{c}
C\\
CA\\
CA^2\\
\vdots\\
CA^{N-1}
\end{array}\right].
\end{equation}
Note that the $ij$-th element of the matrix $A^k$ is free iff, in $\mathcal{G}^T$, there exists at least a path of length $k$ from node $v_i$ to node $v_j$. Hence, the matrix $O$ admits a straightforward interpretation in terms of paths on the graph $\mathcal{G}^T$: the $i$-th element of each row of the matrix $O$, that is, $(c_jA^k)_i$, is nonzero iff, in $\mathcal{G}^T$ there exists at least a path of length $k$ from the $j$-th sensor to the node $v_i$. As there can be multiple paths, say $L$, of length $k$ from $v_j$ to $v_i$, we have that
\begin{equation}\label{eq:path_weight}
(c_jA^k)_i = \sum_{l=1}^L w^l_{ij}(k)
\end{equation}
where the superscript $l$ accounts for the multiplicity of the paths. Eq. \eqref{eq:path_weight} links each column of $O$, say column $j$, to the network node $v_j$ as each of its elements is a sum of weights of the paths to node $v_j$. We anticipate that performing elementary row transformations on the matrix $O$, as will be done in what follows, destroys the interpetation of its elements as weights of paths on a graph but maintains the link between columns of the matrix $O$ and network nodes.

According to \emph{Interpretation 1} of eq. \eqref{eq:net_eq}, rank$(O)$ defines the dimension of the observable subsystem. If we shift to interpretation 2, and consider eq. \eqref{eq:net_eq} as the dynamics of a network, in Theorem \ref{thm:max_o} we will show that $rank(O)$ does not coincide with the number of observable nodes. This is the reason for which we distinguish between a high dimensional system (Interpretation 1) and a linear dynamical network (Interpretation 2), a distinction that may seem subtle, but is indeed crucial when discussing the concepts of controllability and observability. 
\newline We conclude this section by introducing some additional notation. We will denote by
\begin{itemize}
\item $q$, with a slight abuse of notation, both the dimension of the controllable subspace of the pair $(A,B)$ and the dimension of the orthogonal complement of the non observable subspace of the pair $(A,C)$. We will rely on the context to clarify whether we refer to the former or to the latter;
\item $e_i$ the $N$-dimensional versor having a single nonzero entry in its $i$-th position;
\item $\mathcal{N}$ the canonical basis of the network state space, that is, the basis composed of the elements $\lbrace e_i \ \rbrace_{i=1}^N$;
\item $\mathrm{span}(\mathcal{S})$ the linear span of the set of versors $\lbrace e_i:v_i \in \mathcal{S} \rbrace$ with $\mathcal{S}$ any arbitrary set of nodes;
\item $|\mathcal{S}|$ the cardinality of the set $\mathcal{S}$;
\item The symbol $\overline{\mathcal{S}}$ denotes the complement to $\mathcal{V}$ of the set $\mathcal{S}$. 
\item $I_p$ the $p$-dimensional identity matrix.
\end{itemize}

\section{Node Controllability and Observability}
Considering that each network node is a dynamical system of its own (Interpretation 2), here, we define the concepts of node controllability and observability.
\begin{mydef}\label{def:contr}
A node $v_j$ of the dynamical network in eq. \eqref{eq:net_eq} is controllable iff it is possible to steer the value of its state $x_j$ from any initial condition to any target value with a suitable selection of the control signals $u$ in finite time.
\end{mydef}
Note that Definition \ref{def:contr} is coherent with the definition of Structural State Variable Controllability given in \cite{blackhall2010structural}.
\begin{mydef}\label{def:obsv}
A node $v_j$ of the dynamical network in eq. \eqref{eq:net_eq} is observable iff it is possible to reconstruct the value of its state $x_j$ from knowledge of the control signals $u$ and of the measured states $y$ of the sensor nodes.
\end{mydef}
Definitions \ref{def:contr} - \ref{def:obsv} are a direct consequence of the fact that we define a dynamical network as a set of interconnected dynamical systems, the nodes. If one accepts these definitions, then their natural extension to the case of a set of nodes is the following. 
\begin{mydef}\label{def:contr_obs_set}
The set of controllable (observable) nodes $\mathcal{C}$ ($\mathcal{O}$) is defined as the maximal set of nodes that are simultaneously controllable (observable).
\end{mydef}
While definitions \ref{def:contr}-\ref{def:contr_obs_set} are in some sense obvious from a conceptual standpoint, they hide a crucial subtlety from a theoretical standpoint: if the pair $(A,B)$ is not completely controllable, or dually the pair $(A,C)$ is not completely observable, then the Kalman decomposition only allows one to define a set of controllable (observable) state variables $z_i \ i=1,\dots,N$ in a transformed coordinate system. Unfortunately, as definitions \ref{def:contr}-\ref{def:contr_obs_set} refer to the node state variables $x_i$, we cannot evaluate node controllability (observability) after performing a coordinate transformation, as the transformed state variables $z_i$ would not correspond anymore to the network nodes. Hence, finding the mathematical conditions that allows to verify which network nodes are controllable and observable according to definitions \ref{def:contr}-\ref{def:contr_obs_set} is not straightforward. 

Clearly the question arises of which tools can be directly borrowed from systems theory and which, instead, need to be developed for the purpose. The following proposition, provides the first step in answering this question.
\begin{myprop}\label{lem:reach}
The following two facts hold true:
\begin{itemize}
\item[(i)] $|\mathcal{C}|$ always coincides with the dimension of the controllable subspace of the pair $(A,B)$ in eq. \eqref{eq:net_eq};
\item[(ii)] the set $\mathcal{C}$ is not unique.
\end{itemize}
\end{myprop}
\begin{proof} We start by proving fact (i). Denote by $x_{\mathcal{C}}$ the stack vector of the state of the nodes in $\mathcal{C}$, and by $x_{\bar{\mathcal{C}}}$ the stack vector of the remainder of the network nodes. According to definitions \ref{def:contr} and \ref{def:contr_obs_set}, for the nodes of the set $\mathcal{C}$ to be controllable, given any assigned value of their states, say $\bar{x}_{\mathcal{C}}$, there must exist an assignment $\bar{x}_{\bar{\mathcal{C}}}$ of the vector $x_{\bar{\mathcal{C}}}$ such that $[\bar{x}_{\mathcal{C}} \ \bar{x}_{\bar{\mathcal{C}}}]$ defines a point in the controllable subspace of the pair $(A,B)$.

Take the basis, say $\mathcal{T}$, of the controllable subspace of the pair $(A,B)$ that maximizes the number $p$ of versors $e_i$ in the basis. Stacking together the $q$ column vectors encompassed in $\mathcal{T}$, and relabeling the network nodes accordingly (which can be done without loss of generality), we can build the matrix
\begin{equation}\label{eq:bas_to_mat}
\left[ \begin{array}{cc}
I_p & 0_{p\times q-p}  \\
0_{N-p\times p} & F_{ N-p\times q-p}
\end{array}\right],
\end{equation}
where each column of the block $F_{ N-p\times q-p}$ encompasses at least two nonzero entries, as otherwise additional versors $e_i$ could be included in $\mathcal{T}$.
Completing the matrix in eq. \eqref{eq:bas_to_mat} with $(N-q)$ additional columns that ensure the resulting matrix 
\begin{equation}\label{eq:contr_trasf_th}
T = \left[ \begin{array}{ccc}
I_p & 0_{p\times q-p} &0_{p\times N-q} \\
0_{N-p\times p} & F_{ N-p\times q-p} &R_{N-p\times N-q} 
\end{array}\right]
\end{equation}
is full rank\footnote{Note that the matrix $T$ is square by design.}, we obtain a controllability transformation $z = T^{-1}x$. As $T$ is block diagonal, then the matrix $T^{-1}$ has the structure
\begin{equation}\label{eq:T-1}
\left[ \begin{array}{cc}
I_p & 0_{p\times N-p} \\
0_{N-p\times p} & [F_{ N-p\times q-p} \ R_{N-p\times N-q}]^{-1}
\end{array}\right],
\end{equation}
where, in general, the block $[F_{ N-p\times q-p} \ R_{N-p\times N-q}]^{-1}$ is not diagonal as $[F_{ N-p\times q-p} \ R_{N-p\times N-q}]$ is not diagonal. Then, consider any vector $\bar x \in \mathbb{R}^N$ and subdivide it into three subvectors, i.e., $\bar x = [\bar x_p \ \bar x_{q-p} \ \bar x_{N-q}]^T$, where the subscripts denote the dimensions of each subvector. For $\bar x$ to define a point of the controllable subspace of the pair $(A,B)$, $\bar z=T^{-1}\bar x$ must have the structure $[\bar z_q \ 0 ]^T$, with $\bar z_q$ free to take any arbitrary value. Hence, from the structure of the matrix \eqref{eq:T-1}, we can conclude that the entries of $\bar x_p$ can be arbitrarily selected as well as that of $\bar x_{q-p}$, although fixing the latter forces to select the entries of $\bar x_{N-q}$ so to ensure that $\bar z= [\bar z_q \ 0 ]^T$ and thus, fact (i) holds true.
\newline Proving fact (ii) only requires noting that the selection of which node state variables to include in the subvector $\bar x_{q-p}$ (the entries of which can be arbitrarily selected) and which in $\bar x_{N-q}$ (the entries of which must be chosen to ensure $\bar z= [\bar z_q \ 0 ]^T$) is not unique, as the block $[F_{ N-p\times q-p} \ R_{p\times N-q}]^{-1}$ is not diagonal.
\end{proof}
Proposition \ref{lem:reach} implies that we can define a maximal set of controllable nodes $\mathcal{C}$, that is, a set of nodes whose state can be arbritrarily imposed starting from any initial condition and through an appropriate selection of the control signals $u$. As a result, the state of another set of nodes, is driven to a final value that cannot be arbitrarily imposed.
\begin{mydef}
We denote by $\mathcal{P}$ the set of perturbed nodes, that is, the set of nodes whose final values $\bar x_j \ \forall j:v_j \in \mathcal{P}$ are imposed when reaching a target state $\bar x_i \ \forall i: v_i \in \mathcal{C}$, 
\end{mydef}
Proposition 1 allows us to intoduce a definition of the controllable subspace of a complex network. 
\begin{mydef}
The network controllable subspace is $\mathrm{span}(\mathcal{C})$. 
\end{mydef}
\begin{myrem}Note that while the controllable subspace of a dynamical system is unique, from the non uniqueness of $\mathcal{C}$ proved in Proposition 1 we have the non uniqueness of the controllable subspaces of a complex network. Moreover, while the controllable subspace of a dynamical system identifies the directions along which the forced dynamics are confined, this is no longer true for the network controllable subspace.
\end{myrem}

Proposition \ref{lem:reach} states that $|\mathcal{C}|$ is equal to the dimension of the controllable subspace of the pair $(A,B)$. Nevertheless, it also states that there can be multiple different choices of the set $\mathcal{C}$, a fact that has been rarely exploited in the literature. Most existing works, see e.g. \cite{LiSlBa:12,LiGaSo}, rely on Hosoe's theorem \cite{Ho:80}, to find the set $\mathcal{C}$. As this theorem was designed to find the dimension of the controllable subspace of a dynamical system, applying it to complex networks only allows one to find one of the possibly multiple sets $\mathcal{C}$.

Now, we will turn our attention to node observability, a property which we will show cannot be treated through duality.
\begin{mythm}\label{thm:max_o}
The following three statements hold true:
\begin{itemize}
\item[i.] The maximal number of observable nodes $|\mathcal{O}|$ of a network always coincides with the largest number of elements $e_i$ of the basis $\mathcal{N}$ that are orthogonal to the non observable subspace of the pair $(A,C)$;
\item[ii.] the set of observable nodes $\mathcal{O}$ is unique;
\item[iii.] the set of observable nodes $\mathcal{O}$ is generic, in the sense that it does not vary depending on the nonzero entries of the matrix $A$, except for a set of Lebesgue measure zero.
\end{itemize}
\end{mythm}
\begin{proof}
i. The Kalman observability decomposition of the pair $(A,C)$ allows one to find the maximal set of tranformed state variables whose state can be reconstructed from the available measurements. These variables are obtained through a linear transformation $z=Tx$ where $T$ is a matrix of dimension $q\times N$ and its rows form span the $q$-dimensional orthogonal complement of the non observable subspace of the pair $(A,C)$. Then, it is possible to perform elementary row transformations on the matrix $T$ and permute its columns until we obtain the transformation
\begin{equation}\label{eq:T_tilde}
\tilde T = \left[ \begin{array}{cc}
I_p & 0_{p\times N-p} \\
0_{q-p\times p} & F_{ q-p\times N-p} 
\end{array}\right]. 
\end{equation}
where $p$ is the largest integer such that the first $p$ rows of the matrix $\tilde T$ are elements of $\mathcal{N}$. Hence, there exist $p$ node state variables $x_p$ that can be reconstructed from the first $p$ components of the vector $z$. Then, by definition of the integer $p$, no other element of $\mathcal{N}$ can be included in a matrix obtained from $T$ through elementary row transformations and thus is orthogonal to the non observable subspace of the pair $(A,C)$. Hence, no other node state variables can be extracted from the remainder $(q-p)$ elements of $z$.
\newline ii. We will prove this statement by contradiction. From statement i. we know that for a node $v_i$ to be observable, the versor $e_i$ must be orthogonal to the non observable subspace of the pair $(A,C)$. Now consider the set of linearly independent vectors composed of the rows of the matrix $\tilde T$ in eq. \eqref{eq:T_tilde}. As any vector orthogonal to the non observable subspace of the pair $(A,C)$ can be obtained from the rows of $\tilde T$ by means of elementary row transformations, it must be possible to extract $e_i$ from $\tilde T$. Still, this would be a contradiction as $e_i$ is orthogonal to each of the rows of the block $[I_p \ 0_{p\times N-p}]$, and cannot be obtained as a linear combination of the rows of the block $[0_{q-p\times p} \ F_{ q-p\times N-p}]$ from the definition of the scalar $p$.
\newline iii. To prove this statement, we start by noting that whether an element $e_i$ of the basis $\mathcal{N}$ is, or is not, orthogonal to the non observable subspace of the pair $(A,C)$ depends on the linear dependencies between the rows of the observability matrix. For structured matrices, from Lemma \ref{lem:gen_an} we know that any set of rows of a structured matrix are linearly independent if they each encompass an independent entry. While indeed, Lemma \ref{lem:gen_an} ignores the linear dependencies introduced by the powers of $A$ in the observability matrix, as was noted in \cite{ShBP:75}, these linear dependencies are dictated by the positions of the fixed and free entries of the matrix $A$ and thus are generic as well, thus ensuring the genericity of the set $\mathcal{O}$ and proving statement iii.
\end{proof}
Based on Theorem \ref{thm:max_o}, we can propose a definition of the observable subspace for a complex network alternative to the classic definition which holds for Interpretation 1 of eq. \eqref{eq:net_eq}.
\begin{mydef}
The set of observable network states is $\mathrm{span}(\mathcal{O})$, that is, the linear span of the maximum number of elements $e_i$ of the basis $\mathcal{N}$ that are orthogonal to the non observable subspace of the pair $(A,C)$. 
\end{mydef}
\begin{myrem}
An important difference is that while the set of vectors that span the orthogonal complement of the non observable subspace of a dynamical system does not necessarily define an invertible transformation $z_o = T^ox$, $\mathrm{span}(\mathcal{O})$ does.
\end{myrem}
\section{A decomposition of the network nodes}
Given the results in Proposition \ref{lem:reach} and Theorem \ref{thm:max_o}, we propose the following decomposition of the nodes of a linear dynamical network:
\begin{equation}\label{eq:decomp}
\left\lbrace
 \begin{array}{ccc}
\mathcal{C}\cap \mathcal{O}, & \mathcal{P} \cap \mathcal{O}, & \overline{ \mathcal{C}\cup \mathcal{P}} \cap \mathcal{O} \\ 
\mathcal{C}\cap \overline{\mathcal{O}}, & \mathcal{P} \cap \overline{\mathcal{O}}, & \overline{ \mathcal{C}\cup \mathcal{P}} \cap\overline{\mathcal{O}}
\end{array} 
\right\rbrace
\end{equation}
where $\mathcal{C}$ is the selected set of controllable nodes, $\mathcal{O}$ is the (unique) set of observable nodes and $\mathcal{P}$ is the set of perturbed nodes, that is, the nodes in the downstream of the drivers that are not in $\mathcal{C}$. Note that as the set $\mathcal{C}$ is not unique, so is the set $\mathcal{P}$. Substituting to $\mathcal{C}$, $\mathcal{O}$, and $\mathcal{P}$ the subspaces $\mathrm{span}(\mathcal{C})$, $\mathrm{span}(\mathcal{O})$, $\mathrm{span}(\mathcal{P})$ we obtain the decomposition of the network state space associated to the the partition of the network nodes in eq. \eqref{eq:decomp}.

Now, the question arises of how the sets $\mathcal{C}$, $\mathcal{P}$, and $\mathcal{O}$ can be computed. Let us start by showing how to compute the unique set of observable nodes $\mathcal{O}$. To this aim, define the matrix $Q$ as the matrix obtained by stacking together the first $q$ linearly independent rows of the observability matrix $O$. By permuting its columns, the matrix $Q$ can be decomposed as follows:
\begin{equation}\label{eq:Q_b}
Q = \left[ \begin{array}{cc}
H & F
\end{array}\right],
\end{equation}
where $H$ is a ${q\times q}$ full rank matrix, and the dimension of $F$ follows. Algorithm \ref{alg:observability} provides a way to find the set of observable nodes $\mathcal{O}$.
\begin{algorithm}[ht!]
\caption{}
\label{alg:observability}
\begin{algorithmic}
\State $k=1$
\State $Q^k = Q$, $H^k = H$, $F^k = F$, $q_k=q$.
\While{$F^k\neq \mathbf{0}$}
\State Compute $f_k = rank(F^k).$
\Repeat
\State Perform linear combinations of the rows of $Q^k$ 
\Until{the following structure is obtained:
\begin{equation}\label{eq:first_dec}
Q^k_b = \left[ \begin{array}{cc}
H_b^{k1} & F^{k1}\\
H_b^{k2} & F^{k2}
\end{array}\right],
\end{equation}
where $F^{k2}$ has full row rank equal to $f_k$ and $F^{k1} = \mathbf{0}$.}
\Repeat 
\State permute the first $q_k$ columns of the matrix $Q^k_b$ in eq. \eqref{eq:first_dec}
\Until{the following structure is obtained:
\begin{equation*}
\left[ \begin{array}{cc}
H_c^{k1}  &  H_c^{k3}\\
H_c^{k2}  &  H_c^{k4}
\end{array}\right],
\end{equation*}
where $H_c^{k1}$ must be a square full rank block of dimension $q_k-f_k$.}
\State $H^{k+1} \Leftarrow H^{k1}_c$;
\State $F^{k+1} \Leftarrow [H^{k3}_c \ F^{k1}]$;
\State $Q^{k+1} \Leftarrow [H^{k+1} \ F^{k+1}]$;
\State $k \Leftarrow k+1$;
\EndWhile
\State $\bar k = k-1$;
\State \textbf{The observable nodes are those associated to the columns of $H_c^{\bar k1}$}. 
\end{algorithmic}
\end{algorithm}
\begin{mythm}\label{thm:alg_pf}
Algorithm 1 is able to identify the set $\mathcal{O}$.
\end{mythm}
\begin{proof}
Denote by $\bar k$ the iteration at which Algorithm 1 ends, and by $\bar Q$ the matrix obtained from the matrix $Q$ in eq. \eqref{eq:Q_b} through the elementary transformations performed by Algorithm 1, that is,
\begin{equation}\label{eq:bar_Q}
\bar Q = \left[ \begin{array}{cccccc}
H_c^{\bar k 1} &  \mathbf{0} &\mathbf{0} &\mathbf{0} &\mathbf{0} &\mathbf{0}\\
H_c^{\bar k 2}  &  H_c^{\bar k4}  &F^{\bar k 2} &\mathbf{0} &\mathbf{0} &\mathbf{0}\\
\vdots               & \vdots	&  \vdots  	& \vdots            &    \vdots  &\vdots\\
\multicolumn{2}{c}{H_c^{k 2}} & H_c^{k 4}     & F^{k2} &\mathbf{0} &\mathbf{0}\\
\vdots               & \vdots	&  \vdots  	& \vdots            &    \vdots &\vdots\\
\multicolumn{4}{c}{H_c^{12}} & H_c^{14}     & F^{12}
\end{array}\right],
\end{equation}
where the symbol $\mathbf{0}$ stands to denote a zero matrix of suitable dimensions. Hence, its rows span the orthogonal complement of the observable subspace of the pair $(A,C)$. From Theorem \ref{thm:max_o}, we know that the observable network nodes are defined by the maximal set of elements of $\mathcal{N}$ that are orthogonal to the non observable subspace of the pair $(A,C)$. Hence, we must show that the block 
\begin{equation}\label{eq:part_span}
[H_c^{\bar k 1} \ H_c^{\bar k 3} \ F^{\bar k 1}] = [H_c^{\bar k 1} \ \mathbf{0} \ \mathbf{0}].
\end{equation}
of $\bar Q$ defines a set, say $\mathcal{H}$, of elements of $\mathcal{N}$, and that such set is maximal. The first task is trivial, as the block in eq. \eqref{eq:part_span} is composed of the full rank matrix $H_c^{\bar k 1}$ and of zero matrices. Hence, we now prove that no other element of $\mathcal{N}$ can be extracted from $\bar Q$ and added to $\mathcal{H}$. Indeed, this is not possible as in any block $[H_c^{k2} \ H_c^{k4} \ F^{k2} \ \mathbf{0} \ \mathbf{0}]$ with $k<\bar k$, both $H_c^{k4}$ and  $F^{k2}$ have full row rank by design, and thus, regardless of the elementary transformations performed on the matrix $\bar Q$, any of the rows of $\bar Q$ except for those of the block in eq. \eqref{eq:part_span} would encompass at least two nonzero elements. Hence,  the thesis follows.
\end{proof}
From classical systems theory, we know that if a dynamical system is not completely observable one can perform the Kalman transformation to obtain the observable subsystem. Indeed, this transformation is not unique. Still, if one takes the viewpoint of Interpretation 2 and aims at reconstructing the state of the nodes in the set $\mathcal{O}$, then amongst all possible alternatives, one must select a matrix $T$ that defines a transformation $z=Tx$ such that $z_i = x_i$ for all $i$ such that $v_i\in \mathcal{O}$. Algorithm 1 provides the fundamental block of such transformation $T$ as explained in the following remark.
\begin{myrem}
To obtain the transformation matrix $T$ such that $z_i = x_i$ we can take the matrix 
\begin{equation}\label{eq:first_transf}
\left[ \begin{array}{c}
\bar Q \\
\bar Q^{\perp}
\end{array}\right],
\end{equation}
where $Q^{\perp}$ is selected to ensure that the resulting matrix be full rank. Then, considering that the matrix in eq. \eqref{eq:first_transf} can be decomposed as 
\begin{equation*}
\left[ \begin{array}{ccc}
H^{\bar k1}_c \ \mathbf{0} \ \mathbf{0} \\
\quad  R \quad
\end{array}\right],
\end{equation*}
and as $H^{\bar k1}_c$ is full rank, we can perform elementary row operations on the rows of the matrix in eq. \eqref{eq:first_transf} to obtain the transformation
\begin{equation*}
T = \left[ \begin{array}{ccc}
\mathrm{I}_{|\mathcal{O}|} \ \mathbf{0} \ \mathbf{0} \\
\quad  R \quad
\end{array}\right],
\end{equation*}
where we have that $z_i = x_i$ for all $i$ for which $v_i\in \mathcal{O}$.
\end{myrem}
Having provided the tools to compute the set of observable nodes $\mathcal{O}$ we will turn our attention to the sets of controllable nodes $\mathcal{C}$ and perturbed nodes $\mathcal{P}$. Recall that from Proposition \ref{lem:reach} the set $\mathcal{C}$ is not unique. Hence, rather than a tool, we will provide an algebraic condition that must be verified for a set of nodes to be a suitable selection of the set $\mathcal{C}$. To do so, consider the controllability matrix
\begin{equation}\label{eq:contr_mat}
K = [B  \ AB \ A^2B \ \dots \ A^{N-1}B].
\end{equation}
As was the case for the observability matrix $O$, also the matrix $K$ admits a straightforward interpretation in terms of paths on the graph $\mathcal{G}$: the $i$-th element of the $j$-th column of the block $A^kB$ of the matrix $K$ is nonzero iff, in $\mathcal{G}$, there exists at least a path from the $j$-th driver to the node $v_i$. Hence, each row of $K$, say row $i$, is associated to a network node $v_i$. Performing elementary columns transformations on the matrix $K$ does note destroy this association. 

Given this premise, consider the matrix obtained by stacking together the first $rank(K)$ linearly independent columns of the matrix $K$. Then, permute its rows to obtain the following decomposition 
\begin{equation}\label{eq:R}
\left[ \begin{array}{c}
H\\
F
\end{array}\right],
\end{equation}
where $H$ is square and full rank. By performing elementary column transformations on the matrix in eq. \eqref{eq:R} (leveraging for instance Algorithm \ref{alg:observability}) one can obtain a matrix having the structure
\begin{equation}\label{eq:bar_R_b}
\left[ \begin{array}{cc}
I_h & 0\\
0 & R
\end{array}\right],
\end{equation}
where $h$ is the maximal number of elements of the basis $\mathcal{N}$ that can be included in a basis of the linear span of the columns of the matrix $K$, and as $I_h$ has dimension $h$, the dimensions of the full column rank matrix $R$ follow. Then, through additional elementary operations on its columns, and permuting its rows, we can turn the matrix in eq. \eqref{eq:bar_R_b} into the form
\begin{equation}\label{eq:bar_R_c}
\left[ \begin{array}{cc}
I_h &0\\
0 & R^{22}\\
0 &R^{32} 
\end{array}\right],
\end{equation}
where the matrix $R^{22}$ is square full rank. Note that the decomposition performed in eq. \eqref{eq:bar_R_c} of the matrix in eq. \eqref{eq:bar_R_b} is not unique, as there can be multiple ways of building the blocks $R^{22}$ and $R^{32}$. We must stress that given our premise, bulding these blocks in different ways leads to including into them elementary tranformations of different rows of $R$ and thus to associating to $R^{22}$ and $R^{32}$ different sets of nodes. By completing the matrix in eq. \eqref{eq:bar_R_c} with $N-rank(K)$ columns such that the resulting matrix,
\begin{equation}\label{eq:T_contr}
T = \left[ \begin{array}{ccc}
I_h &0 & 0\\
0 & R^{22} &0\\
0 &R^{32}  &T^{33}
\end{array}\right],
\end{equation}
is full rank\footnote{Note that the matrix $T$ is square by design.}, we obtain a controllability transformation $z = T^{-1}x$ where
\begin{equation}\label{eq:T_inv}
T^{-1} = \left[ \begin{array}{ccc}
I_h &0 & 0\\
0 & (R^{22})^{-1} &0\\
0 &\star  &(T^{33})^{-1}
\end{array}\right],
\end{equation}
and the block
\begin{equation}\label{eq:star}
\star = -(T^{33})^{-1}R^{32}(R^{22})^{-1}.
\end{equation}
Now, consider any vector $\bar x \in \mathbb{R}^N$ and, consistently with the structure of $T^{-1}$, subdivide it into three subvectors, i.e., $\bar x = [\bar x^T_h \ \bar x^T_{rank(K)-h} \ \bar x^T_{N-rank(K)}]^T$, where the subscripts denote their dimensions. Then, $\bar x$ is a reachable point of the network state space if $\bar z=T^{-1}\bar x$ has the structure $[\bar z^T_{rank(K)} \ 0^T_{N-rank(K)} ]^T$. Hence, we can conclude that in order to specify any reachable point
\begin{itemize}
\item the state of the nodes associated to the first $h$ columns of the matrix in eq. \eqref{eq:T_inv} can be arbitrarily selected\footnote{Note that when performing the inverse of the matrix $T$ the associations between rows of $T$ and network nodes, become associations between columns of $T^{-1}$ and network nodes, consistently with the equation $z=T^{-1}x$.};
\item the state of the nodes associated to the columns of the matrix in eq. \eqref{eq:T_inv} corresponding to the block $(R^{22})^{-1}$ can be arbitrarily selected; 
\item the selection of the state of the last $N-rank(K)$ network nodes must be made fulfilling the constraint
$$\star\, \bar x_{rank(K)-h} + (T^{33})^{-1}x_{N-rank(K)} = 0$$ 
which, from the expression of $\star$ in eq. \eqref{eq:star} implies that
\begin{equation}\label{eq:pert_node_cond}
\bar x_{N-rank(K)} =R^{32}(R^{22})^{-1} \bar x_{rank(K)-h}
\end{equation}
\end{itemize}
 The aforementioned arguments constitute the theoretical basis for the following proposition.
\begin{myprop}
The set of controllable network nodes $\mathcal{C}$ is obtained by the union of two subsets, say $\mathcal{C}^1$ and $\mathcal{C}^2$, where
\begin{itemize}
\item $\mathcal{C}^1$ is the set of $h$ nodes associated to the first $h$ rows of the matrix in eq. \eqref{eq:bar_R_c}, and is unique.;
\item $\mathcal{C}^2$ is composed of $(rank(K)-h)$ nodes associated to the rows of the block $R^{22}$ of the matrix in eq. \eqref{eq:bar_R_c} and thus, as the selection of $R^{22}$ is not unique, so is $\mathcal{C}^2$. 
\end{itemize}
The set $\mathcal{P}$ is composed of the nodes associated to the components of $\bar x_{N-rank(K)}$ that become generically different from zero through eq. \eqref{eq:pert_node_cond}. As the selection of $R^{32}$ depends on that of $R^{22}$, and as the latter is not unique, then $\mathcal{P}$ is not unique either.
\end{myprop}

\section{Example}
As an example of application of Algorithm \ref{alg:observability}, we consider a network of $N=8$ nodes with dynamics
\begin{align}\label{eq:net_example}
&\dot x = \left[ \begin{array}{cccccccc}
0 & 0 & 0 & 0 & 3 & 0 & 0 & 7 \\ 
0 & 0 & 0 & 0 & 0 & 0 & 0 & 0 \\ 
0 & 0 & 2 & 0 & 0 & 0 & 0 & 0 \\ 
0 & 8 & 0 & 0 & 0 & 0 & 0 & 0 \\ 
0 & 0 & 1 & 1 & 0 & 0 & 0 & 0 \\
0 & 0 & 0 & 0 & 0 & 0 & 0 & 0 \\ 
0 & 0 & 0 & 0 & 0 & 0 & 0 & 0 \\ 
0 & 0 & 0 & 0 & 0 & 5 & 4 & 1
\end{array}  \right]x \\\nonumber
& y = \left[ \begin{array}{cccccccc}
1 & 0 & 0 & 0 & 0 & 0 & 0 & 0
\end{array}  \right]x 
\end{align}
which corresponds to the graph shown in Fig. \ref{fig:net_example}.
\begin{figure}[ht!]
\centering
\includegraphics[scale=0.5,trim={0cm 7cm 0cm 8cm}]{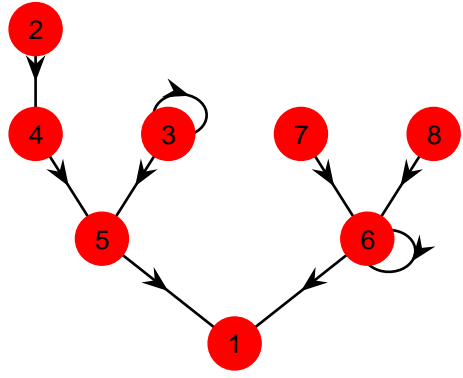}
\caption{Graph of the network in eq. \eqref{eq:net_example}.}
\label{fig:net_example}
\end{figure}
The observability matrix for the network in eq. \eqref{eq:net_example} is
\begin{equation}\label{eq:ex_obs_mat}
O = \left[ \begin{array}{cccccccc}
1 & 0 & 0 & 0 & 0 & 0 & 0 & 0 \\ 
0 & 0 & 0 & 0 & 3 & 0 & 0 & 7 \\ 
0 & 0 & 3 & 3 & 0 & 35 & 28 & 7 \\ 
0 & 24 & 6 & 0 & 0 & 35 & 28 & 7 \\ 
0 & 0 & 12 & 0 & 0 & 35 & 28 & 7 \\ 
0 & 0 & 24 & 0 & 0 & 35 & 28 & 7 \\ 
0 & 0 & 48 & 0 & 0 & 35 & 28 & 7 \\ 
0 & 0 & 96 & 0 & 0 & 35 & 28 & 7
\end{array}  \right],
\end{equation}
and we have that $q=rank(O)=6$. As in this example the matrix $C$ has only one row, we know from the Cayley-Hamilton theorem that the first $6$ rows of $O$ are linearly independent, and thus form the matrix $Q$. Given these premises, we can now use Algorithm 1 to find the set $\mathcal{O}$ of the network observable nodes. First, we set $Q^1=Q$, and note that the first six columns of the matrix $Q_1$ define a block with full rank. Then, we set
\begin{equation}\label{eq:H^1}
H^1= H = \left[ \begin{array}{cccccc}
1 & 0 & 0 & 0 & 0 & 0\\ 
0 & 0 & 0 & 0 & 3 & 7\\ 
0 & 0 & 3 & 3 & 0 & 7\\ 
0 & 24 & 6 & 0 & 0 & 7\\ 
0 & 0 & 12 & 0 & 0 & 7\\ 
0 & 0 & 24 & 0 & 0 & 7
\end{array}  \right], \quad 
F^1 = F = \left[ \begin{array}{cc}
0 & 0 \\ 
0 & 0 \\ 
28 & 35 \\ 
28 & 35 \\ 
28 & 35 \\ 
28 & 35
\end{array}  \right],
\end{equation}
As $f_1=1$, we perform the following elementary transformations on the rows $r_1,\dots,r_6$ of $Q^1= [H^1 \ F^1]$:
\begin{enumerate}
\item $r_3 \Leftarrow r_3-r_6$;
\item $r_4 \Leftarrow r_4-r_6$;
\item $r_5 \Leftarrow r_5-r_6$;
\end{enumerate}
to obtain the matrix 
\begin{equation}\label{eq:Q^1_b}
Q^1_b = \left[ \begin{array}{cccccccc} 
1 &0 &0 &0 &0 &0 &0 &0\\
0 &0 &0 &0 &3 &7 &0 &0\\
0 &0 &-21 &3 &0 &0 &0 &0\\
0 &24 &-18 &0 &0 &0 &0 &0\\
0 &0 &-12 &0 &0 &0 &0 &0\\
0 &0 &24 &0 &0 &7 &28 &35
\end{array}\right].
\end{equation}
Moreover, as $f_1 = 1$, we have that
\begin{equation}\label{eq:H1}
\left[\begin{array}{c}
H^{11}_b\\
H^{12}_b
\end{array} \right]
= \left[ \begin{array}{cccccc}
1 &0 &0 &0 &0 &0\\
0 &0 &0 &0 &3 &7\\
0 &0 &-21 &3 &0 &0\\
0 &24 &-18 &0 &0 &0\\
0 &0 &-12 &0 &0 &0\\
0 &0 &24 &0 &0 &7
\end{array}\right].
\end{equation}
As no column of the matrix in eq. \eqref{eq:H1} has its first five elements equal to zero, the exit condition of Algorithm 1 is not verified for $k=1$. Hence, we set $k\Leftarrow 2$ and
\begin{equation}\label{eq:Q2}
Q^2 = [H^2  \ F^2] \Leftarrow 
\left[ \begin{array}{cccccc}
1 &0 &0 &0 &0 &0\\
0 &0 &0 &0 &3 &7\\
0 &0 &-21 &3 &0 &0\\
0 &24 &-18 &0 &0 &0\\
0 &0 &-12 &0 &0 &0\\
\end{array}\right].
\end{equation}
As $f_2 = 1$, we perform the following elementary row tranformations on the matrix $Q^2$:
\begin{enumerate}
\item $r_5 \Leftarrow r_5+r_2$;
\item $r_2 \Leftarrow r_2-r_5$;
\end{enumerate}
to obtain the matrix
\begin{equation}\label{eq:Q2_b}
Q^2_b = \left[ \begin{array}{cccccc}
1 &0 &0 &0 &0 &0\\
0 &0 &12 &0 &0 &0\\
0 &0 &-21 &3 &0 &0\\
0 &24 &-18 &0 &0 &0\\
0 &0 &-12 &0 &3 &7
\end{array}\right].
\end{equation}
From the matrix in eq. \eqref{eq:Q2_b}, we can extract the block
\begin{equation}\label{eq:H2_c}
H^{21}_b = \left[ \begin{array}{cccccc}
1 &0 &0 &0 \ \ \vline &0 \\
0 &0 &12 &0 \ \ \vline  &0 \\
0 &0 &-21 &3 \ \ \vline &0 \\
0 &24 &-18 &0 \ \ \vline &0
\end{array}\right],
\end{equation}
in which the vertical line highlights that no column permutations are required to obtain $H^{23}_c= \mathbf{0}$ and thus the exit condition of Algorithm 1 is verified. 

Hence, while the rank of the observability matrix is equal to 6, the set of observable nodes is composed of only four nodes, that is, $\mathcal{O} = \lbrace v_1,  \ v_2, \ v_3, \ v_4\rbrace$. Then, to allow observing the state of the nodes of $\mathcal{O}$ we perform, on the rows of the matrix,
$$\bar Q = \left[\begin{array}{cccc}
H^{21}_c &0 &0 & 0\\
H^{22}_c &H^{24}_c &F^{22} &0\\
\multicolumn{2}{c}{H_c^{12}} &H^{14}_c & F^{12}
\end{array} \right]
$$
the following elementary transformations
\begin{enumerate}
\item $r_2  \Leftarrow (1/16)(r_2+(2/3)r_4)$;
\item $r_4 \Leftarrow r_4+(1/3)r_3 - 24 r_2$;
\item $r_3 \Leftarrow (1/54)(r_3- 3r_4)$,
\item $r_4 \Leftarrow r_4 + 25r_3$;
\item $r_5 \Leftarrow r_5 +12r_3$;
\item $r_6 \Leftarrow r_6 -24r_3$;
\end{enumerate}
thus obtaining the matrix 
\begin{equation}\label{eq:bar_Q}
\bar Q = \left[ \begin{array}{ccccccccc} 
1 &0 &0 &0 &\vline &0 &0 &0 &0\\
0 &1 &0 &0 &\vline &0 &0 &0 &0\\
0 &0 &1 &0 &\vline &0 &0 &0 &0\\
0 &0 &0 &1 &\vline&0 &0 &0 &0\\
\hline
0 &0 &0 &0 &\vline &3 &7 &0 &0\\
0 &0 &0 &0 &\vline &0 &7 &28 &35
\end{array}\right].
\end{equation}
Then, we complete the matrix $\bar Q$ with two additional rows that ensure the resulting matrix is full rank thus providing the Kalman observability transformation,
\begin{equation}\label{eq:bar_Q}
T = \left[ \begin{array}{cccccccc} 
1 &0 &0 &0 &0 &0 &0 &0\\
0 &1 &0 &0 &0 &0 &0 &0\\
0 &0 &1 &0 &0 &0 &0 &0\\
0 &0 &0 &1 &0 &0 &0 &0\\
0 &0 &0 &0 &3 &7 &0 &0\\
0 &0 &0 &0 &0 &7 &28 &35\\
0 &0 &0 &0 &4 &-3 &2 &4\\
0 &0 &0 &0 &5 &2 &4 &1
\end{array}\right].
\end{equation}
As the transformed state is $z=Tx$, the reader will notice that the first four rows of the matrix $T$ indicate that $z_1=x_1$, $z_2 = x_2$, $z_3 = x_3$, and $z_4=x_4$, and thus from knowledge of the transformed state variables $z_i \ i=1,\dots, 4$ one can reconstruct the values of the node state variables $x_i \ i=1,\dots, 4$. 

\section{Conclusions}
When the number of driver nodes is not sufficient to make a network completely controllable and that of the sensor nodes is not sufficient to reconstruct the state of all nodes, one should identify the nodes that can be driven to a desired state (the set ${\mathcal C}$ in this paper) and the set of the observable nodes (${\mathcal O}$). In this way each network node can be labeled as either controllable and observable, controllable but not observable, and so on, thus giving an idea of the achievable control goals for a given network configuration. When dealing with linear network dynamics described by the triplet $(A,B,C)$ one is tempted to apply the Kalman decomposition to solve this problem, making reference to the theory of structural controllability/observability.
Following this approach, the results in this paper lead to some unexpected and somehow surprising conclusions that can be summarized as follows:
\begin{itemize}
\item[-] the sets ${\mathcal C}$ and ${\mathcal O}$ are not directly obtainable from the reachability and observability system matrices but some sophisticated manipulation is required; 
\item[-] the labeling of the controllable network nodes cannot be done in a unique way. Indeed, although the number of nodes in ${\mathcal C}$ is always constant and equal to the rank of the system reachability matrix, different nodes of the same network can be selected to become members of the set ${\mathcal C}$;
\item[-] once we select a set ${\mathcal C}$ among all the possible alternatives, an additionals set of nodes can be identified, the set ${\mathcal P}$. These are the nodes that will be perturbed by the control signals and dragged to a nonzero (but known) value during the control action;
\item[-] contrarily to the case of the controllable nodes, it is impossible to state that the number of observable nodes of a network coincides with the number of observable states of the dynamical system described by the same triplet of matrices.
\end{itemize}

We have provided an algorithm that, based on algebraic manipulations of the observability and reachability system matrices, can be used to identify the sets ${\mathcal C}$, ${\mathcal O}$ and ${\mathcal P}$ thus enabling a partition of the network nodes using a proper combination of the following labels: ``controllable", ``perturbed'', ``observable"  ``uncontrollable",``unperturbed" and ``unobservable". 
This partition induces what we called the \textit{network state space decomposition}. 
Playing a similar role to that of the Kalman decomposition for the control of dynamical systems, but inherently different from it, this partition of the network nodes provides essential information for the control of a network.

\bibliographystyle{IEEEtran}
\bibliography{biblio}
\end{document}